 \documentclass[a4paper,11pt,twoside,reqno]{amsart} 
\DeclareSymbolFontAlphabet{\mathcal}{symbols}
\usepackage[colorlinks=true,linkcolor=blue]{hyperref}

\usepackage[cmtip,all]{xy}
\usepackage{hyperref}
\usepackage{amsmath}
\usepackage[psamsfonts]{amssymb}
\usepackage[psamsfonts]{eucal}
\usepackage{amsthm}
\usepackage[psamsfonts]{amssymb}
\usepackage[scaled]{helvet}
\usepackage{mathrsfs}
\usepackage{framed}
\usepackage{bm}
\usepackage[a4paper]{geometry}
\usepackage{enumerate}
\usepackage{graphicx}
\usepackage{tikz}
\usepackage{color}
\definecolor{trama}{gray}{.875}
\usepackage{lmodern}
\usepackage{textcomp}
\usepackage{enumerate}
\usepackage{wrapfig}
    \title{Symmetric Lie models of a Triangle}

\author{Urtzi Buijs}
\address{Departamento de \'Algebra, Geometr\'ia y Topolog\'ia\\
         Universidad de M\'alaga\\
        Ap. 59\\
         29080-M\'alaga\\
         Espa\~na}
\email{ubuijs@uma.es}

\author{Yves F\'elix}
\address{Institut de Math\'ematiques et Physique\\
         Universit\'e Catholique de Louvain-la-Neuve\\
         Louvain-la-Neuve\\
         Belgique}
\email{Yves.felix@uclouvain.be}

\author{Aniceto Murillo}
\address{Departamento de \'Algebra, Geometr\'ia y Topolog\'ia\\
         Universidad de M\'alaga\\
        Ap. 59\\
         29080-M\'alaga\\
         Espa\~na}
         \email{aniceto@uma.es}

\author{Daniel Tanr\'e}
\address{D\'epartement de Math\'ematiques, UMR 8524\\
         Universit\'e de Lille~1\\
         59655 Villeneuve d'Ascq Cedex\\
         France}
\email{Daniel.Tanre@univ-lille1.fr}

\date{\today}

\thanks{
The first and third authors have been partially supported by the Junta de Andaluc\'\i a grant FQM-213.
The fourth author has been partially supported by the  ANR-11-LABX-0007-01  ``CEMPI''.
The  authors have been partially supported by the MINECO grants MTM2013-41768-P and MTM2016-78647-P} 


\theoremstyle{plain}

\newtheorem{proposition}{Proposition}[section]
\newtheorem{theoremb}[proposition]{Theorem}
\newtheorem{lemma}[proposition]{Lemma}

\theoremstyle{definition}
\newtheorem{definition}[proposition]{Definition}

\theoremstyle{remark}

\newcommand{\secref}[1]{Section~\ref{#1}}

\newcommand{\thmref}[1]{Theorem~\ref{#1}}
\newcommand{\propref}[1]{Proposition~\ref{#1}}
\newcommand{\lemref}[1]{Lemma~\ref{#1}}

\def\cG{{\mathcal G}}

\def\1{{\mathbf 1}}

\def\te{{\mathtt e}}

\def\tv{{\mathtt v}}
\def\tw{{\mathtt w}}
\def\tx{{\mathtt x}}

\def\L{\mathbb{L}}

\def\Q{\mathbb{Q}}

\def\Z{\mathbb{Z}}

\def\Ll{\widehat{\mathbb L}}
 
 \newcommand{\lasu}{{\mathfrak{L}}} 
 \def\ad{{\rm ad}}
 
 \def\MC{{\rm MC}}
 
\begin{document}

\begin{abstract}
R. Lawrence and D. Sullivan have constructed a Lie model for an interval from the geometrical
idea of flat connections and flows of gauge transformations. 
Their model supports an action of the symmetric group $\Sigma_2$ reflecting 
the geometrical symmetry  of the interval.
In this work, we present a Lie model of the triangle with an action of the symmetric group 
$\Sigma_3$ compatible with the  geometrical symmetries of the triangle.
We  also prove  that the model of a graph consisting of a circuit with $k$ vertices  
admits a Maurer-Cartan element stable by the automorphisms of the graph.
\end{abstract}

\maketitle

\section*{Introduction}

In rational homotopy theory, the existence of algebraic models of topological spaces
 in terms of Lie algebras was first established by D. Quillen in \cite{Q2} for 
 simply connected spaces. His work is based on a succession of couples of adjoint functors
 between the category of simply connected spaces and the category of differential
  Lie algebras, positively graded.
  The complexity of the Quillen functors contrasts with the simple couple of functors 
  $(A_{PL},\langle -\rangle)$
  introduced by D. Sullivan (\cite{Dennis}) between the category of simplicial sets and 
  the category of
  commutative differential graded algebras (cdga in short).
Sullivan's construction 
is based on a contractible simplicial cdga $\Omega_{\bullet}$ defined by
$$\Omega_{n}=\Q[t_{0},\dots,t_{n}]\otimes \land(dt_{0},\dots,dt_{n})/(\sum t_{i}-1, \sum dt_{i}),$$
with the $t_{i}$'s in degree zero and the $dt_{i}$'s in degree one.
The cdga $A_{PL}(X)$ associated to a simplicial set $X$ is defined from the set of simplicial maps
from $\Omega_{\bullet}$ to $X$.

\medskip
A natural question is the existence of a similar construction in the context of Lie algebras. 
We solve it in \cite{four} from the  cocommutative, coassociative, 
infinity coalgebra structure on the rational singular chains and a transfer process.
For a finite simplicial set, $K$, 
this gives a complete differential graded Lie algebra $\lasu(K)$
which contains information on the rational homotopy type of $K$, 
as the Bousfield-Kan completion of its fundamental group for instance.
The departure point is the construction of a family
$\lasu_k=\lasu(\Delta^k)$ which extends the construction by R. Lawrence and D. Sullivan
of a complete differential Lie algebra $\lasu_1$ in the case of the interval, see \cite{LS}.
We recall it in the first section.
Their model is symmetric for the action of the symmetric group $\Sigma_2$
deduced from the geometry of the interval.

\medskip
In \secref{sec:triangle}, we recall the construction of
a Lie model $\lasu_2$ of the triangle $\Delta^2$.
Based on a vertex of $\Delta^2$, this model does not reflect the symmetries of
a triangle. 
We modify it in a Lie model
with an action of the symmetric group 
$\Sigma_3$ compatible with the different symmetries.
For that, we consider the geometrical action of $\Sigma_3$ on the generators of $\lasu_2$
and prove the existence of a Maurer-Cartan element 
stable by any element of $\Sigma_3$. 
Even though the differential in $\lasu_2$ is, up to some extend, uniquely determined 
(\propref{prop:unicitylasu2}), there are many Maurer-Cartan elements invariant by $\Sigma_3$.  

\medskip
The Maurer-Cartan elements play the role of points 
and   are recalled in \secref{sec:back}. Therefore, a Maurer-Cartan element invariant by the action 
of $\Sigma_{3}$ is an algebraic representation of the barycentre of the triangle.
After the interval of Lawrence and Sullivan, this is the second stage in the research of a Lie $\Sigma_{n}$-model
of a simplex $\Delta^n$.

\medskip
Our first step consists to prove that  the Lie model $\lasu(\Gamma)$,
 of a graph $\Gamma$ consisting of a circuit with $k$ vertices,
  contains a Maurer-Cartan element invariant by the automorphisms of the graph.  

\medskip
In all this work,  vector spaces and chain complexes are $\Z$-graded and we denote by $|x|$ the degree
of an element. 
The free Lie algebra on a rational vector space $V$
is represented  by $\L(V)$. The vector subspace $\L^i(V)$ is the set of elements
of $\L(V)$ of  bracket length $i$.
The objects considered in this work are complete free Lie algebras,  defined by
$$\Ll(V)=\varprojlim_i \L(V)/\L^{\geq i}(V).$$
Endowed with a Lie differential, they are called
complete differential graded  Lie algebras
(henceforth cdgl) and, if
the explicit elicitation of the vector space $V$ is not necessary, we denote also $L=\Ll(V)$. 
If $x\in L$, the map $\ad_x\colon L\to L$ is the adjoint derivation $y\mapsto [x,y]$.

\section{Background}\label{sec:back}

The \emph{Baker-Campbell-Hausdorff product} (henceforth BCH-product)
of two elements $x$ and $y$ of degree~0 in a free complete Lie algebra $\Ll(V)$ is defined by
$$x\ast y=\log(e^xe^y).$$ 
This product satisfies the following properties (see
\cite{four}, \cite{Gad} or \cite{LS} for instance) for any $x$, $y$, $z$ in $L_{0}$
and $r$, $s$ in $\Q$:
\begin{enumerate}[(i)]
\item $x\ast(y\ast z)= (x\ast y)\ast z$,
\item $(rx)\ast(sx)=(r+s)x$. In particular, the inverse of $x$ is $-x$; i.e., $(x)\ast(-x)=0$,
\item $-(y\ast x)=(-x)\ast(-y)$,
\item $\ad_{x\ast y}=\ad_x \ast \ad_y$
 and $e^{\ad_{x\ast y}}=e^{\ad_x}\circ e^{\ad_y}$,
 \item $e^{\ad_x}(y)=x\ast y\ast (-x)$ and $e^{\ad_x}(y)\ast e^{\ad_x}(z)=e^{\ad_x}(y\ast z)$,
 \item $r(x\ast y)\ast x=x\ast r(y\ast x)$ and
 $r(y\ast x)\ast (-x)=(-x)\ast r(x\ast y)$.
\end{enumerate}
As the BCH-product is associative, we can introduce the following notation, for any family
$(w_i)_{1\leq i\leq r}$ of elements in $L_0$,
$$\ast_{i=1}^rw_i=w_1\ast w_2\ast\dots\ast w_r.$$

\medskip\noindent
The next property follows  from the Dynkin formula for the BCH-product,
see \cite{Serre}.

\begin{lemma}\label{lem:bch}
Let  
$(w_i)_{1\leq i\leq r}$ be elements of degree zero in $\Ll(V)$ that coincide in 
the quotient $\Ll(V)/\Ll^{>n}(V)$.
Then, for any rational numbers $\lambda_{i}$, 
we have an equality between the classes modulo $\Ll^{>n+1}(V)$,
$$[\ast_{i=1}^r \lambda_i w_i]^{<n+1} =  [\sum_{i=1}^r \lambda_i  \, w_{i}]^{<n+1}\,.$$
\end{lemma}

\medskip
Given a cdgl $(L,d)$,  a \emph{Maurer-Cartan element} (in short a \MC-element) is an element $a\in L_{-1}$
such that
\begin{equation}\label{equa:MC}
da+\frac{1}{2}[a,a]=0.
\end{equation}
If $a$ is a \MC-element in a cdgl $(L,d)$, the derivation defined by
$$d_a x=\ad_a(x)+dx,\;x\in L,$$
is a differential.

\medskip
Denote by $\MC(L)$ the set of \MC-elements of $L$. There is an action of the complete 
Lie algebra of elements of degree~0, $L_0$, on $\MC(L)$, called
\emph{gauge action}, and defined as follows:
\begin{equation}\label{equa:gauge}
x\cG a=e^{\ad_x}(a)-\frac{e^{\ad_x}-1}{\ad_x}(dx)=
\sum_{i\geq 0}\frac{\ad^i_x(a)}{i!} -\sum_{i\geq 0}\frac{\ad^i_x(dx)}{(i+1)!},
\end{equation}
with $x\in L_0$, $a\in \MC(L)$, see \cite[Proposition 1.2]{four2}. 
If $L=\Ll(V)$, the linear  and the quadratic parts 
of the gauge action can be determined explicitly  as,
\begin{eqnarray}
(x\cG a)_1&=&a_1-(dx)_1, \label{equa:lineargauge}\\
(x\cG a)_2&=&
a_2+[x_1,a_1]-\frac{1}{2}[x_1,d_1x_1]-d_2x_1-d_1x_2, \label{equa:quadgauge}
\end{eqnarray}
where the decompositions along the bracket length are denoted
$a=\sum_{i\geq 1} a_i$, $x=\sum_{i\geq 1}x_i$,
$d=\sum_{i\geq 1}d_i$ with $d_i(V)\subset \L^i(V)$.

\medskip
The  gauge action is compatible with the BCH-product in $L_0$: for any $x,y\in L_0$ and $a\in \MC(L)$, we have
\begin{equation}\label{equa:gaugebch}
x\cG y\cG a=(x\ast y)\cG a.
\end{equation}
(In the context of model of the interval, this property is a consequence of \cite[Theorem 2]{LS} but it also
appears  in  different works, see \cite[Appendix~B]{MR2876262} for instance.)
We denote by $\widetilde{\MC}(L)$ the set of equivalence classes for the gauge action.

\vspace{5mm}
In \cite{LS},
R. Lawrence and D. Sullivan  have defined a Lie model $\lasu_1$ of  {the interval} $\Delta^1$:
$$\lasu_{1}=(\Ll(a,b,x),d)$$ with $|a|=|b|=-1$, $|x|=0$, 
$da=-\frac{1}{2} [a,a]$, $db=-\frac{1}{2}[b,b]$ and
\begin{eqnarray}
dx&=&\ad_{x}a+\frac{\ad_{-x}}{e^{-\ad_{x}}-1}(b-a)=\frac{\ad_{x}}{1-e^{\ad_{x}}}a + \frac{\ad_{x}}{1-e^{-\ad_{x}}}b
\label{equa:diffinterval}
\\
&=&
\ad_{x}b+\frac{\ad_{x}}{e^{\ad_{x}}-1}(b-a).\label{equa:diffintervalb}
\end{eqnarray}
The cdgl $\lasu_1$ is called the LS-interval.
The elements $a$ and $b$ are \MC-elements.
The set of \MC-elements of $\lasu_1$ has been characterized in \cite{four2}
and the unicity of such model is established in \cite{PPTD}. 
As it was observed by Lawrence and Sullivan, an element $x$ of degree~0 in a cdgl $(\Ll(V),d)$
defines a flow  by
$$\frac{dv}{dt}=d x-\ad_x(v), \;|v|=-1.$$
The flowing by $x$ from $a$ to $b$, in time 1, corresponds to 
\begin{equation}\label{equa:flot}
(-x)\cG a=b
\end{equation}
 which determines
the  expression (\ref{equa:diffinterval}) of the differential of $x$.  
Such $x$ is called a path from $a$ to $b$.
The differential $dx$ 
is clearly stable by the symmetry 
($\tau\colon a \leftrightarrow b$, $\tau x=-x$)
which generates the symmetric group $\Sigma_2$.
Mention also the existence
of a symmetric model of a 2-gon in \cite{Gad}, where the geometrical action of $\Sigma_2$ 
is reflected on the differential of its model. 

\medskip
The next proposition recalls the well known    relation between $d_a$ and $d_b$ in $\lasu_1$. 
We give the proof for completeness.

\begin{proposition}\label{prop:expdiff}
Let $(L,d)$ be a cdgl containing a LS-interval $(\Ll(a,b,x),d)$. Then, for any $w\in L$, 
we have
$$d_ae^{\ad_x}(w)=e^{\ad_x}(d_b w).$$
\end{proposition}

\begin{proof}
From \cite[Lemma~1]{LS}, we have
$$
d_{b}(e^{-\ad_{x}}w)=
e^{-\ad_{x}}(d_{b}w)+
(-1)^{|w|} e^{-\ad_{x}}\ad_{w}\frac{e^{-\ad_{x}}-1}{\ad_{x}}(d_{b}x).
$$
Replacing $d_{b}x$ by its value (\ref{equa:diffintervalb}) gives
\begin{eqnarray*}
d_{b}(e^{-\ad_{x}}w)
&=&
e^{-\ad_{x}}\left(d_{b}w+(-1)^{|w|}\ad_{w}(b-a)\right)\\
&=&
e^{-\ad_{x}}(d_{a}w).
\end{eqnarray*}
\end{proof}
\section{Lie model of a graph  consisting of a circuit}\label{sec:sigmancdgl}
 
 In this section we consider cdgl's endowed with the action of a group $G$.
 
\begin{definition}\label{def:actionlie}
A $G$-cdgl is a cdgl $(L,d)$, with a group action $G\times L\to L$, compatible with the differential and the Lie bracket; i.e.,
for any $v,w \in L$ and $\varphi\in G$, we have
$$\varphi[v,w]=[\varphi(v),\varphi(w)] \text{ and } \varphi(dv)=d(\varphi(v)).$$
\end{definition}

 Let $\Gamma$ be a graph consisting of a circuit with $k$ vertices. 
 Then $\Gamma$ is the union of $k$ intervals and its Lie model is
 a concatenation of LS-intervals (\cite{LS}); i.e.,
$$\lasu(\Gamma)=(\Ll(\oplus_{i=1}^{k}( \Q \tv_i\oplus  \Q \tx_i),d)\,,$$ 
where the $\tv_i$ are \MC-elements and the sub dgl's $(\Ll (\tv_i, \tv_{i+1}, \tx_i),d)$ are LS-intervals. Here 
by convention, we set $\tv_{k+1}=\tv_1$,  $\tx_{k+1}=\tx_1$. 
 
 The group $G$ of automorphisms of $\Gamma$ is the dihedral group $D_{2k}$ 
 generated by the cyclic permutation $\sigma$ and the involution $\tau$ on the first vertex,
 $$G= <\sigma, \tau; \sigma^k, \tau^2, \sigma\tau = \tau \sigma^{-1}>\,.$$
The cdgl $\lasu(\Gamma)$ is a $G$-cdgl. The actions of $\sigma$ and $\tau$ are defined by:
$$\sigma (\tx_i) = \tx_{i+1}\,, \hspace{5mm} \sigma (\tv_i)= \tv_{i+1}\,, \hspace{5mm} 
  \tau (\tv_i)= \tv_{k-i+2}\,, \hspace{5mm}\tau(\tx_i) = -\tx_{k-i+1}\,.$$
 Denote by $V$ the graded vector space generated by the $\tx_i$ and the $\tv_i$. 
 We decompose the differential $d$ in $\lasu(\Gamma)$ in the form 
 $d= \sum_{i\geq 1} d_i$ with $d_i(V)\subset  \mathbb L^i(V)$. 

 \begin{lemma}\label{lem:boucle}
With the above notation, the following properties are verified.
\begin{enumerate} 
\item[(i)] Any $d_1$-cycle of degree 0 in $\Ll^{\geq 2}(V)$ is equal to zero.
\item[(ii)] An element $w$ of degree zero is such that
$w\cG \tv_1=\tv_1$ if  and only if 
$w=\lambda(\tx_1\ast\tx_2\ast \dots \ast\tx_k)$
for some rational number $\lambda$.
\end{enumerate}
\end{lemma}

\begin{proof}
(i)  Since $H(\Ll(V),d_1)=\Ll(H(V,d_{1}))$, we are reduced to the determination of the homology of the
1-dimensional simplicial complex defined by the graph $\Gamma$.
We obtain,
$$H(\Ll(V),d_1)=\Ll(\tv_1,\tx_1+\tx_2+\dots + \tx_k),$$
and any  cycle of degree 0 having a zero linear part is a boundary.
Since $\Ll(V)$
has no element of strictly positive degree, this cycle must be 0.

(ii)
First, from 
(\ref{equa:gaugebch}),
we have $\lambda(\tx_1\ast\cdots \ast\tx_k)\cG\tv_1=\tv_1$.
Now, let $w$ such that $w\cG \tv_1=\tv_1$. Using (\ref{equa:lineargauge}), we get
$\tv_1-d_1w_1=\tv_1$ and $d_1w_1=0$. This implies the existence of $\lambda\in\Q$ such that
$w_1=\lambda(\sum_{i=1}^k \tx_i)$. The element $z=(-\lambda(\tx_1\ast\cdots \ast\tx_k))\ast w$
verifies $z\cG \tv_1=\tv_1$ and has a zero linear part.
From (\ref{equa:quadgauge}), we deduce that its quadratic part verifies $d_1z_2=0$. 
From (i),  this implies $z_2=0$. By induction, we have $z_i=0$ for all $i$ and 
$w=\lambda(\tx_1\ast\dots \ast\tx_k).$ \end{proof}

\begin{lemma}\label{lemme1}
With the previous notation, any
$w\in \Ll( \oplus_{i=1}^k \Q \tx_i)$  such that $\sigma(d_1 w)=d_1 w$ 
is invariant by $\sigma$; i.e., $\sigma(w)=w$.
\end{lemma} 

\begin{proof}
 Since $\sigma$ preserves the bracket length, we can suppose  that $w\in\L^{\ell}(\oplus_{i=1}^k \Q \tx_i)$ for some ${\ell}\geq 1$. In the
enveloping algebra, $T(\oplus_{i=1}^k \Q \tx_i)$, $w$ can be  written in an unique way as
$$w=\sum_{i=1}^{k} \tx_iX_i \text{ with } X_i\in T^{{\ell}-1}(\oplus_{i=1}^{k} \Q \tx_i).$$
Denoting $\nu_{i} = d_1\tx_i=\tv_{i+1}-\tv_i$, we have
$\sum_{i=1}^k \nu_i = 0$ and
\begin{eqnarray*}
d_1w&=& \sum_{i=1}^{k-1}\nu_iX_i-(\sum_{i=1}^{k-1}\nu_i)X_{k}+
\sum_{i=1}^{k}\tx_i\,d_1X_i
=
\sum_{i=1}^{k-1}\nu_i(X_i-X_{k})+
\sum_{i=1}^{k}\tx_i\,d_1X_i,\\
\sigma(d_1w)&=&
\sum_{i=1}^{k-1}\nu_{i}(\sigma X_{i-1}-\sigma X_{k-1}) 
+\sum_{i=1}^{k}\tx_{i+1}\sigma(d_1X_i).
\end{eqnarray*}
The hypothesis $\sigma(d_1 w)=d_1 w$ implies
$$
X_i-X_k=\sigma X_{i-1}-\sigma X_{k-1}.
$$
Thus, the element
$\sigma X_{i-1}-X_i$ does not depend on $i$. Denoting by
$E=\sigma X_{i-1}-X_i$ this common value, we have
$\sigma X_i=X_{i+1}+E$ which implies
$$\sigma(w)=\sum_{i=1}^{k}\tx_{i+1}\sigma X_i=
\sum_{i=1}^{k}\tx_{i+1}X_{i+1}+(\sum_{i=1}^{k}\tx_{i+1})E=w+(\sum_{i=1}^{k}\tx_{i})E.$$
Since $w$ and $\sigma(w)$ are Lie elements, the vector
$(\sum_{i=1}^{k}\tx_{i})E$ must be primitive. Therefore, we have  $E=0$
and $\sigma(w)=w$.
\end{proof}

\begin{theoremb}\label{prop:actionlinear}
 The cdgl $\lasu(\Gamma)$ contains a \MC-element invariant by the actions of $\sigma$ and~$\tau$.
\end{theoremb} 

\begin{proof} The first step of the proof consists in a construction by induction of a \MC-element stable by the action of $\sigma$.
  We define
\begin{equation}\label{equa:depart}
P[1]= -\left(\, \sum_{i=1}^{k-1} \frac{k-i}{k} \tx_i\,\right)\, \cG\tv_1.
\end{equation}
Using the property (\ref{equa:lineargauge}) of the gauge action, 
one has the linear part $P[1]_1$ of $P[1]$:
$$
P[1]_1 =\tv_1+\sum_{i=1}^{k-1}\frac{k-i}{k}d_1\tx_i
=
\frac{1}{k}\sum_{i=1}^{k} \tv_i.$$
Therefore, $P[1]$ is a \MC-element whose linear part is $\sigma$-invariant.
By induction we suppose to have constructed  the \MC-elements, $P[i]$, for $i=1,\dots,n-1$, such that
\begin{enumerate} 
\item[(a)] $P[i]-P[i-1]\in \L^{\geq i}(V)$,
\item[(b)] $\sigma P[i]_j=P[i]_j$, for any $j\leq i$, where $P[i]_j$ denotes the component of $P[i]$ 
in the Lie brackets of length $j$,
\end{enumerate}
As the graph $\Gamma$ is connected, from \cite[Theorem]{four4} there is only one non-trivial
equivalence class of \MC-elements. 
Thus, there is an element of degree 0, $T$, satisfying the equation
\begin{equation}\label{equa:letn}
 T \cG P[n-1] =\sigma P[n-1] .
\end{equation}
Formula (\ref{equa:lineargauge})
gives the following equality between the linear parts,
$$P[n-1]_1-d_1T_1=\sigma P[n-1]_1.
$$
Since  $P[n-1]_1= \sigma P[n-1]_1$, 
 $d_1 T_1=0$ and $T_1$ is
$\sigma$-invariant. We prove by induction on $j$  that $T_j$ is $\sigma$-invariant for $j<n$. We thus suppose $T_1, \dots , T_{r-1}$ to be $\sigma$-invariant. Then from $T\cG P[n-1]= \sigma P[n-1]$, we get that $d_1T_r$ is $\sigma$-invariant, and by Lemma \ref{lemme1}, $T_r$ is $\sigma$-invariant.

The same construction in Lie brackets of length $n$ gives  a
$\sigma$-invariant element, $S$, such that
\begin{equation}\label{equa:quadPnn}
P[n-1]_{n}+S-d_1T_{n}=\sigma P[n-1]_{n}.
\end{equation}
With   successive applications of  $\sigma$, 
we deduce the system of equations,
\begin{eqnarray*}
P[n-1]_{n}+S-d_1T_{n}
&=&
\sigma P[n-1]_{n},\\
\sigma P[n-1]_{n}+S-d_1\sigma T_{n}
&=&
\sigma^2 P[n-1]_{n},\\
\dots &=&\dots,\\
\sigma^{k-1}P[n-1]_{n}+S-d_1\sigma^{k-1}T_{n}
&=&
 P[n-1]_{n}.
\end{eqnarray*}
The adding of these equations gives
\begin{equation}\label{equa:quadPnnn}
k S=d_1 (T_{n} + \sigma T_{n} +\dots +\sigma^{k-1} T_{n} ).
\end{equation}
We define $T'$ and $P'$
by
$$T'_i = \left\{ \begin{array}{ll} T_i \,, & \text{for }i<n,\\
\frac{k-1}{k} T_n + \frac{k-2}{k}\sigma T_n + \dots + \frac{1}{k}\sigma^{k-2}T_n \hspace{5mm}\mbox{} & \text{for }i= n,\\
0 & \text{for }i>n,\end{array}\right.$$
and $$P'= T'\cG P[n-1]\,.$$
By construction,
 $P'_{<n}= P[n-1]_{<n}$.
 On the other hand,
\begin{equation}\label{equa:defPn}
P'_{n}=
P[n-1]_{n} + S' - \sum_{i=0}^{k-2}\frac{k-i-1}{k} d_1\sigma^i T_{n},
\end{equation}
with $S'$ invariant by $\sigma$.
We apply $\sigma$ to this equality and obtain:
$$
\sigma P'_{n}
=
\sigma P[n-1]_{n} +S'-\sum_{i=0}^{k-2}\frac{k-i-1}{k}d_1\sigma^{i+1}T_{n}.$$
Using (\ref{equa:quadPnn}) and extracting the value of $\sigma^{k-1} T_{n}$ 
of (\ref{equa:quadPnnn}), we have
\begin{eqnarray*}
\sigma P'_{n}&=&
P[n-1]_{n}+S-d_1T_{n} +S'-\sum_{i=0}^{k-3}\frac{k-i-1}{k}d_1\sigma^{i+1}T_{n} \\
&&
\hskip 1cm
-\frac{1}{k}\left(
kS-\sum_{i=0}^{k-2}d_1\sigma^iT_{n}\right).
\end{eqnarray*}
With a reordering of the indices, this formula becomes
\begin{eqnarray*}
\sigma P'_{n}&=&
P[n-1]_{n}+S'-\sum_{j=0}^{k-2}\frac{k-j-1}{k}d_1\sigma^{j}T_{n}\\
&=&
P'_{n}.
\end{eqnarray*}
 We write $P[n]= P'$ and the inductive step is closed. Now by property (a) of the induction, 
 the sequence $P[n]$ converges to a \MC-element $P$ with $\sigma P= P$. 

\medskip
 We now use $P$ to construct a \MC-element $\Omega$ invariant by $\sigma$ and  $\tau$. 
   Since $\sigma \tau = \tau \sigma^{k-1}$,   we  deduce that $\tau P$ is also $\sigma$-invariant.
On the other hand, since $P$ and $\tv_1$ are gauge equivalent, there is an  element $\tw$  such that
$\tw\cG P=\tv_1$ and we set
$$\alpha=(-\tau \tw)\ast  \tw;$$
i.e., $\alpha$ can be represented as a composition of paths,
$$\xymatrix{
P && \tau P\ar[ll]_\alpha \\
&\tv_1= \tau \tv_1 \ar[ul]^{\tw} \ar[ur]_{\tau \tw}.&
}$$
From  (\ref{equa:flot})
and from Property (iii) of the BCH-product, recalled in \secref{sec:back}, we get
\begin{eqnarray}
\alpha\cG P
&=& \tau P,\\
\tau\alpha
&=&
(- \tw)\ast \tau  \tw=-((-\tau  \tw)\ast  \tw)\nonumber=
-\alpha. \label{equa:alphaequi}
\end{eqnarray}
We define a \MC-element $\Omega$ by 
\begin{equation}\label{equa:centregravite}
\Omega=\left(\frac{\alpha}{2}\right)\cG P
\end{equation}
and deduce that
$$\tau\Omega=\left(-\frac{\alpha}{2}\right)\cG\tau P=\frac{\alpha}{2}\cG(-\alpha)\cG\tau P=
\frac{\alpha}{2}\cG P=\Omega.$$
Thus the \MC-element $\Omega$ is $\tau$-invariant. We prove now that $\alpha$ is $\sigma$-invariant
which implies the $\sigma$-invariance of $\Omega$.

\smallskip
 The 
linear part of  $\alpha\cG P=\tau P$ is
$P_1-d_1\alpha_1=\tau P_1.$
This gives the $\sigma$-invariance of $d_1\alpha_1$ and by \lemref{lem:boucle}(i) the $\sigma$-invariance of
$\alpha_1$. 
Suppose that   the components $\alpha_j$ of $\alpha$,
in the Lie brackets of length $j$, are $\sigma$-invariant for any $j\leq n$. The component  of the equality
$\alpha\cG P=\tau P$ in the Lie bracket of length $(n+1)$ gives
$$P_{n+1}+S-d_1\alpha_{n+1}=(\tau P)_{n+1},$$
where  $S$ is $\sigma$-invariant by induction.  
This implies the $\sigma$-invariance of $d_1\alpha_{n+1}$ and 
therefore of $\alpha_{n+1}$ as required.
 \end{proof}

\section{Symmetric Lie model of a  triangle}\label{sec:triangle}

Let $\Delta^2$ be the simplicial complex of the triangle, with 3 vertices and 3 edges. In
the spirit of the model of the interval (\cite{LS}), we are looking for a model of the triangle
as a cdgl whose generators are the simplices of $\Delta^2$, whose linear part of the differential
is the differential of the associated simplicial complex and whose the sub-cdgl's corresponding to edges are
LS-intervals.
We call it \emph{a Lie model of $\Delta^2$.}
The next result proves existence and unicity of such model.

\begin{proposition}\label{prop:unicitylasu2}
Let $\Delta^2$ be the simplicial complex of the triangle.
\begin{enumerate}[1)]
\item
A Lie model of  $\Delta^2$ is  the cdgl
$$
\lasu_2=(\Ll(\tv_{1},\tv_{2},\tv_{3},\tx_1,\tx_{2},\tx_{3},\te),d),$$
with $\vert \tv_i\vert = -1$, $\vert \tx_i\vert = 0$ and $\vert  \te\vert = 1$.
The elements $\tv_i$ are \MC-elements, the intervals $[\tv_1, \tv_2, \tx_1]$, $[\tv_2, \tv_3, \tx_2]$ 
and $[\tv_3, \tv_1, \tx_3]$ are LS-intervals, and  
$$
 d\te = [\te,\tv_{1}]+
 \tx_{1}\ast \tx_{2}\ast \tx_{3}.
$$
\item
Let $\lasu_2$ be as before and
$(\Ll(\tv_{1},\tv_{2},\tv_{3},\tx_1,\tx_{2},\tx_{3},\te),d')$
be a cdgl such that 
\begin{enumerate}[(i)]
\item the differentials $d$ and $d'$ coincide on the generators $\tv_i$ and $\tx_i$,
\item the linear parts  $d_1\te$ and $d'_1\te$ coincide, 
\item $d'_{\tv_1}(\te)\in \Ll(\tx_1,\tx_2,\tx_3)$.
\end{enumerate}
Then, we have $d=d'$.
\end{enumerate}
\end{proposition}

\begin{proof}
1) The proof is reduced to: ``$\tx_{1}\ast \tx_{2}\ast \tx_{3}$ is a $d_{\tv_{1}}$-cycle.''
Let $\lasu_{1}=(\Ll(a,b,x),d)$ be a LS-interval. The composition of morphisms as in \cite[Theorem 2]{LS}
gives a cdgl map
$$\varphi\colon \lasu_{1}\to (\Ll(\tv_{1},\tv_{2},\tv_{3},\tx_1,\tx_{2},\tx_{3}),d),$$
with $\varphi(a)=\varphi(b)=\tv_{1}$ and $\varphi(x)=\tx_{1}\ast\tx_{2}\ast\tx_{3}$. We deduce
\begin{eqnarray*}
d(\tx_{1}\ast\tx_{2}\ast\tx_{3})
&=&
d\varphi(x)=\varphi(dx)\\
&=&
\varphi\left(\ad_{x}b+\frac{\ad_{x}}{e^{\ad_{x}}-1}(b-a)
\right)\\
&=&
\varphi[x,b]=-\ad_{\tv_{1}}(\tx_{1}\ast\tx_{2}\ast\tx_{3}).
\end{eqnarray*}
This implies $d_{\tv_{1}}(\tx_{1}\ast\tx_{2}\ast\tx_{3})=0$ as expected.

\medskip
2) We decompose $d_{\tv_1}$ and $d'_{\tv_1}$ along the bracket size as
$d_{\tv_1}=\sum_{j\geq 1}d_j$
and
$d'_{\tv_1}=\sum_{j\geq 1}d'_j$. By hypothesis, we have $d_1=d'_1$.
Suppose, by induction, that $d_j=d'_j$ for any $j<n$. From $d^2={d'}^{2}=0$, we deduce
$$d_1d'_n=-\sum_{j=2}^n d_j d'_{n+1-j}=-\sum_{j=2}^nd_j d_{n+1-j}=d_1d_n.$$
Therefore, the element $d_n\te-d'_n\te$ is a $d_1$-cycle of degree 0 in $\Ll(\tx_1,\tx_2,\tx_3)$.
From the first assertion of \lemref{lem:boucle}, we deduce $d_n\te=d'_n\te$.
\end{proof}

Denote by $\Gamma$ the boundary of the triangle and by $\lasu({\Gamma})$ its Lie model.  
As the symmetric group $\Sigma_{3}$ coincides with the dihedral group $D_{6}$, we may consider the action of 
$\Sigma_{3}$ on $\lasu({\Gamma})$ described at the beginning of \secref{sec:sigmancdgl} and apply 
\thmref{prop:actionlinear}.

\begin{theoremb}\label{prop:center} Let $\Omega$ be a $\Sigma_3$-invariant \MC-element in $\lasu(\Gamma)$. Then 
the cdgl $\lasu_2$ is a $\Sigma_3$-cdgl isomorphic to the cdgl 
$(\lasu(\Gamma) \,\widehat{\coprod}\, \L (\te'),d)$   satisfying the following properties:
\begin{enumerate}[(i)]
\item  for some $\beta$ of degree 0, we have
 $$d\te'=-[\Omega,\te']+\beta\ast\tx_1\ast \tx_2\ast\tx_3\ast (-\beta),$$
\item
 $\sigma(\te')=\te'\text{ and }\tau(\te')=-\te'$.
\end{enumerate}
\end{theoremb}

\begin{proof} 
Let $\beta$ be a path from $\Omega$ to $\tv_{1}$ and set $\te'= e^{\ad_\beta}(\te)$.
From \propref{prop:expdiff}, we get
\begin{eqnarray*}
d_\Omega \te'
&=&
d_{\Omega}e^{\ad_{\beta}}(\te)=e^{\ad_{\beta}}(d_{\tv_{1}}\te)\\
&=&
 \beta \ast \tx_1\ast \tx_2\ast \tx_3\ast (-\beta).
\end{eqnarray*}
Let us observe that $\sigma (\beta) \ast (-\tx_{1})$ and $\beta$ are two paths from $\Omega$ to $\tv_1$. From
\lemref{lem:boucle}(ii) applied to 
$(-\beta)\ast \sigma (\beta) \ast (-\tx_{1})$, there exists  $\lambda\in\Q$ such that 
$$\sigma(\beta)=\beta\ast \lambda(\tx_{1}\ast \tx_{2}\ast \tx_{3})\ast \tx_{1}.$$
Replacing $\sigma(\beta)$ by its value in $\sigma(d_{\Omega}\te')$, we get 
\begin{eqnarray*}
\sigma d_\Omega (\te')
&=&
\sigma(\beta)*\tx_{2}*\tx_{3}*\tx_{1}*(-\sigma(\beta))\\
&=&
\beta* \lambda(\tx_{1}*\tx_{2}*\tx_{3})*\tx_{1}*\tx_{2}*\tx_{3}*\tx_{1}*(-\tx_{1})*(-\lambda)(\tx_{1}*\tx_{2}*\tx_{3})*(-\beta)\\
&=&
\beta*\tx_{1}*\tx_{2}*\tx_{3}*(-\beta)\\
&=&
d_\Omega (\te')
\end{eqnarray*}

\medskip
On the other hand, let us recall
$\tau(\tx_{2})=-\tx_{2}$, $\tau(\tx_{1})=-\tx_{3}$, $\tau(\tv_{1})=\tv_{1}$ and $\tau(\tv_{2})=\tv_{3}$. 
Applying
\lemref{lem:boucle}(ii)  to 
$(-\beta)\ast \tau (\beta)$, we obtain   $\mu\in\Q$ such that
$$\tau(\beta)=\beta\ast \mu(\tx_{1}\ast\tx_{2}\ast\tx_{3}).$$
Once again, a similar computation gives
\begin{eqnarray*}
\tau (d_\Omega (\te')) 
&=&
\beta*\mu(\tx_{1}\ast\tx_{2}\ast\tx_{3})*(-\tx_{3})*(-\tx_{2})*(-\tx_{1})*(-\mu)(\tx_{1}*\tx_{2}*\tx_{3})*(-\beta)\\
&=&
\beta*(-(\tx_{1}*\tx_{2}*\tx_{3}))*(-\beta)=
-d_\Omega (\te').
\end{eqnarray*}
\end{proof}

\noindent {\bf Problem.} 
In \thmref{prop:actionlinear}, we prove that the Lie model of a graph consisting of a circuit  contains a \MC-element 
invariant along the group of automorphisms of the graph. 
Is this result true in general: does the Lie model of any finite graph  have 
a \MC-element invariant along the automorphisms of the graph?

\providecommand{\bysame}{\leavevmode\hbox to3em{\hrulefill}\thinspace}
\providecommand{\MR}{\relax\ifhmode\unskip\space\fi MR }
\providecommand{\MRhref}[2]{%
  \href{http://www.ams.org/mathscinet-getitem?mr=#1}{#2}
}
\providecommand{\href}[2]{#2}

\end{document}